\documentclass[draft]{amsart}

\usepackage{amsmath}
\usepackage{amsthm}
\usepackage{amssymb}
\usepackage{url}

\usepackage{enumerate}
\usepackage[mathscr]{eucal}
\usepackage{mathrsfs}
\usepackage{verbatim}

\theoremstyle{plain}
\newtheorem{theorem}{Theorem}[section]
\newtheorem{lemma}[theorem]{Lemma}
\newtheorem{proposition}[theorem]{Proposition}
\newtheorem{corollary}[theorem]{Corollary}
\newtheorem{claim}[theorem]{Claim}

\theoremstyle{remark}
\newtheorem{remark}[theorem]{Remark}
\newtheorem{remarks}[theorem]{Remarks}

\numberwithin{equation}{section}
\numberwithin{theorem}{section}

\newcommand{\supp}{\operatorname{supp}}
\newcommand{\dist}{\operatorname{dist}}

\makeatletter
\@namedef{subjclassname@2010}{%
  \textup{2010} Mathematics Subject Classification}
\makeatother

\begin{document}


\date{}

\title[Lattice Points]{Lattice points in large convex planar domains of finite type}

\author[J. Guo]{Jingwei Guo}

\address{Jingwei Guo\\Department of Mathematics\\ University of Wisconsin-Madison\\Madison,
WI 53706, USA}
\email{guo@math.wisc.edu}

\thanks{}

\begin{abstract}
Let $\mathcal{B}$ be a compact convex planar domain with smooth boundary of finite type and $\mathcal{B}_\theta$ its rotation by an angle $\theta$. We prove that for almost every $\theta\in[0, 2\pi]$ the remainder $P_{\mathcal{B}_\theta}(t)$ is of order $O_{\theta}(t^{2/3-\zeta})$ with a positive number $\zeta$ independent of the domain.
\end{abstract}

\maketitle


\section{Introduction}

If $\mathcal{B}$ is a compact domain in $\mathbb{R}^2$, the number of lattice points $\mathbb{Z}^2$ in the dilated domain $t\mathcal{B}$ is approximately $\textrm{area}(t\mathcal{B})$ and the lattice point problem is to estimate the remainder, $P_{\mathcal{B}}(t)$, in the equation
\begin{equation*}
P_{\mathcal{B}}(t)=\#(t\mathcal{B}\cap\mathbb{Z}^2)-\textrm{area}(\mathcal{B})t^2 \quad \textrm{for $t\geqslant 1$}.
\end{equation*}
One has $P_{\mathcal{B}}(t)=O(t)$ since the measure of the boundary of the dilated domain is $O(t)$. See \cite{kratzel} for the history and fundamental results and methods of this problem.

If $\mathcal{B}$ has sufficiently smooth boundary with nonzero curvature the standard estimate is $P_{\mathcal{B}}(t)=O(t^{2/3})$. With various sophisticated methods this bound has been improved and the best known bound is due to Huxley~\cite{huxley2003}. See \cite{huxley} for an introduction to his method. Notice that the conjecture $P_{\mathcal{B}}(t)=O(t^{1/2+\varepsilon})$ is still open.

If we weaken the curvature condition on the boundary, the remainder may become much larger. For instance if the boundary is of finite type $\omega$, $\omega \geqslant 3$, ({\it i.e.} the maximal order of vanishing of the curvature is $\omega-2$), Colin de Verdi\`{e}re~\cite{colin} showed that
\begin{equation*}
P_{\mathcal{B}}(t)=O(t^{1-1/\omega}).
\end{equation*}
At an earlier time Randol~\cite{randol} proved the same bound for a particular domain $\{(x_1, x_2): x_1^{\omega}+x_2^{\omega}\leqslant 1\}$ with $\omega\geqslant 4$ an even integer. Furthermore, he showed that the exponent is the best possible. The sharpness of this bound for this particular domain is because that normals at boundary points with curvature zero are parallel to the coordinate axes. If we consider $\mathcal{B}_\theta$, the rotation of $\mathcal{B}$ by an angle $\theta\in[0, 2\pi]$ about the origin, however, we expect a substantially better estimate for most choices of $\theta$. Colin de Verdi\`{e}re~\cite{colin} showed that if $\theta$ is a sufficiently irrational angle satisfying a certain Diophantine condition then $P_{\mathcal{B}_\theta}(t)=O_{\theta}(t^{2/3})$. Moreover, he showed that
\begin{equation}
P_{\mathcal{B}_\theta}(t)=O_\theta(t^{2/3}) \quad \textrm{for a.e. $\theta$}. \label{bound1}
\end{equation}
Tarnopolska-Weiss~\cite{Tarnopolska} obtained the same bound for almost every rotation of a planar domain of finite type which is star-like with respect to the origin.

Iosevich~\cite{iosevich} further developed this type of results by weakening the curvature condition, and proved
\begin{equation*}
P_{\mathcal{B}_\theta}(t)=O_\theta(t^{2/3}\log^{\delta(p_0)}(t))  \quad \textrm{for a.e. $\theta$},
\end{equation*}
with $\delta(p_0)>1/p_0$ for a given $p_0>1$, for certain class of convex planar domains whose curvature is allowed to vanish to infinite order. His result was extended, in the paper \cite{brandolini} by Brandolini, Colzani, Iosevich, Podkorytov, and Travaglini, to arbitrary convex planar domains with no curvature or regularity assumption on the boundary.

One can also develop Colin de Verdi\`{e}re's result in another direction--to improve the exponent $2/3$ under certain curvature conditions. The first result of this kind can be found in M\"{u}ller and Nowak~\cite{mullernowak}, where they considered a compact planar domain bounded by a closed smooth Jordan curve determined by an analytic function. They evaluated the contributions of boundary points with curvature zero to the remainder $P_{\mathcal{B}}(t)$ and distinguished the cases where the tangent at such a point has rational or irrational slope.
Assuming certain Diophantine approximation of irrational slope, they gave an asymptotic formula for $P_{\mathcal{B}}(t)$, whose main term is given explicitly by absolutely convergent Fourier series, with an error term of order $O(t^{\gamma})$, $\gamma<2/3$ unspecified (see also Nowak~\cite{nowak G1, nowak G2}). As a consequence they obtained
\begin{equation*}
P_{\mathcal{B}_\theta}(t)=O_{\theta}(t^{\gamma}) \quad \textrm{for a.e. $\theta$},
\end{equation*}
where the $\gamma<2/3$ depends on the order of vanishing of the curvature, but not on $\theta$.

Later on, M\"{u}ller and Nowak~\cite{mullernowak2} improved their previous results by using the discrete Hardy--Littlewood method. In particular they obtained
\begin{equation}
|P_{\mathcal{B}_\theta}(t)|\leqslant C_{\theta}\max(t^{\gamma(\omega)}, t^{7/11}(\log t)^{45/22}) \quad \textrm{for a.e. $\theta$}, \label{bound2}
\end{equation}
where $\gamma(\omega)=2/3-1/(9\omega-12)$ and $\omega-2$ is the maximal order of vanishing of the curvature.

Note that $\gamma(\omega)$ tends to $2/3$ as $\omega$ goes to infinity. The goal of this paper is to prove a bound $P_{\mathcal{B}_\theta}(t)=O_{\theta}(t^{2/3-\zeta})$ for almost every rotation with a $\zeta>0$ that does not depend on $\omega$. We will give up some information in terms of the Diophantine condition as given in those papers by M\"{u}ller and Nowak. Specifically we obtain:

\begin{theorem}\label{ae-theorem}
Let $\zeta=1/3831$. If $\mathcal{B}$ is a compact convex planar domain with smooth boundary of finite type which contains the origin as an interior point, then
\begin{equation*}
P_{\mathcal{B}_\theta}(t)=O_{\theta}(t^{2/3-\zeta}) \quad \textrm{for a.e. $\theta$}.
\end{equation*}
\end{theorem}

This bound is better than \eqref{bound1}, and it is better than \eqref{bound2} if $\omega>427$. Theorem \ref{ae-theorem} follows easily from the following theorem, which contains an improved but more technical statement.

\begin{theorem}\label{finitetype} Let $\zeta=1/3831$. If $\mathcal{B}$ is a compact convex planar domain with smooth boundary of finite type $\omega$ which contains the origin as an interior point, then
\begin{equation*}
\sup\limits_{t\geqslant 2} \log^{-b}(t)t^{-2/3+\zeta+\sigma(\omega)}|P_{\mathcal{B}_\theta}(t)| \in L^1(S^1),
\end{equation*}
where $b>1$ and
\begin{equation*}
\sigma(\omega)=\frac{832}{1277(1277\omega-2496)}.
\end{equation*}
\end{theorem}

In this paper we focus on the remainder for rotated planar domains. For other interesting related results (like the mean square lattice point discrepancy for rotated planar domains, or the remainder for rotated high dimensional domains), see \cite{I-S-S}, \cite{I-S-S-2}, \cite{randol_R^n}, \cite{I-S-S-0}, etc.

{\it Notations:} We use the usual Euclidean norm $|x|$ for a point $x\in \mathbb{R}^2$. $B(x, r)\subset \mathbb{R}^2$ represents the Euclidean ball centered at $x$ with radius $r$. The norm of a matrix $A\in \mathbb{R}^{2\times 2}$ is given by $\|A\|=\sup_{|x|=1}|Ax|$. We set $e(f(x))=\exp(-2\pi i f(x))$, $\mathbb{Z}_{*}^{2}=\mathbb{Z}^{2}\setminus \{0\}$, and $\mathbb{R}^2_*=\mathbb{R}^2\setminus \{0\}$. The Fourier transform of any function $f\in L^1(\mathbb{R}^2)$ is $\widehat{f}(\xi)=\int f(x) e(\langle x, \xi\rangle) \, dx$.

We fix $\chi_0$ to be a smooth cut-off function whose value is $1$ on $B(0, 1/2)$ and $0$ on the complement of $B(0, 1)$. For a set $E\subset \mathbb{R}^2$ and a positive number $a$, we define $E_{(a)}$ to be the larger set
\begin{equation*}
E_{(a)}=\{x\in \mathbb{R}^2: \dist(E, x)<a \}.
\end{equation*}

We use the differential operators
\begin{equation*}
D^{\nu}_{x}=\frac{\partial^{|\nu|}}{\partial x_1^{\nu_1} \partial x_2^{\nu_2}} \quad \left(\nu=(\nu_1, \nu_2)\in \mathbb{N}_0^2, |\nu|=\nu_1+\nu_2 \right)
\end{equation*}
and the gradient operator $\nabla_x$. We often omit the subscript if no ambiguity occurs.

For functions $f$ and $g$ with $g$ taking nonnegative real values, $f\lesssim g$ means $|f|\leqslant Cg$ for some constant $C$. If $f$ is nonnegative, $f\gtrsim g$ means $g\lesssim f$. The Landau notation $f=O(g)$ is equivalent to $f\lesssim g$. The notation $f\asymp g$ means that $f\lesssim g$ and $g\lesssim f$.

{\it Structure of the paper:} After giving some geometric facts related to convex planar domains of finite type, we will study the support function of such domains in \S \ref{non-vanishing}. We will establish the nonvanishing of certain $2\times 2$ determinants, which allows us to use the method of stationary phase later in the estimate of some exponential sums. This result is a two dimensional refinement to M\"uller~\cite{mullerII} Lemma 3 with more precise bounds given. We will then prove in \S \ref{four-tran} our main analytic tool, the asymptotic formula of the Fourier transform of certain indicator functions. In \S \ref{proof-claims} we give an estimation of certain exponential sums. One important difference between all these results in this paper and their analogues in the nonvanishing curvature case is that all bounds here contain curvature terms explicitly. In \S \ref{mainterm}, we put all these ingredients together to prove our main theorem. In the appendix we collect several standard results mainly from the oscillatory integral theory.

\section{Some Geometric Facts}\label{geometry}

In the rest of this paper, unless otherwise stated $\mathcal{B}$ will always denote a compact convex planar domain with smooth boundary of finite type $\omega$. In particular we assume, only in \S \ref{mainterm}, that it contains the origin as an interior point.

Since $\partial \mathcal{B}$ is compact and of finite type, it is easy to see that $\partial \mathcal{B}$ can contain only finitely many points with curvature zero. Assume $\{P_i\}_{i=1}^{\Xi}$ are all such points and the curvature of $\partial \mathcal{B}$ at $P_i$ vanishes of order $\omega_i-2$. Each $\omega_i$ must be an even integer greater than three due to the convexity of $\mathcal{B}$.

The Gauss map of $\partial \mathcal{B}$, denoted by $\vec{n}$, maps each boundary point $x\in \partial \mathcal{B}$ to a unit exterior normal $\vec{n}(x)\in S^1$. It is bijective since $\mathcal{B}$ is convex and of finite type. At each boundary point with nonzero curvature, there exists a small neighborhood on which the Gauss map is a diffeomorphism. Denote by $\vec{t}(x)$ the unit tangent vector at $x\in \partial \mathcal{B}$ such that $\{\vec{t}(x), -\vec{n}(x)\}$ has the same orientation as $\{e_1, e_2\}$.

When we express $\partial \mathcal{B}$ by a parametric equation, we always assume the orientation is counterclockwise. Hence the signed curvature is always nonnegative.

For each nonzero $\xi\in \mathbb{R}^2$, there exists a unique point $x(\xi)\in \partial \mathcal{B}$ where the exterior normal is $\xi$. Denote by $K_{\xi}$ the curvature of $\partial \mathcal{B}$ at $x(\xi)$. Denote the $2\times 2$ rotation matrix by
\begin{equation*}
R_{\theta}=\begin{pmatrix}
\cos \theta & -\sin \theta\\ \sin \theta & \cos \theta
\end{pmatrix}
\end{equation*}
and its transpose by $R_{\theta}^{t}$. Then $\mathcal{B}_\theta=R_{\theta}\mathcal{B}$. Define $x^{\theta}(\xi)=R_{\theta}x(R_{\theta}^{t}\xi)$ and $K_{\xi}^{\theta}=K_{R_{\theta}^{t}\xi}$. Then $x^{\theta}(\xi)$ is the unique point on $\partial \mathcal{B}_{\theta}$ where the exterior normal is $\xi$ and $K_{\xi}^{\theta}$ is the curvature of $\partial \mathcal{B}_{\theta}$ at $x^{\theta}(\xi)$.

If $v_1$ and $v_2$ are vectors $\in \mathbb{R}^2$, denote by $\mathfrak{A}_{v_1, v_2}$ the angle between them that is in $[0, \pi]$. By Taylor's formula, it is easy to prove

\begin{lemma}\label{lemma2:1} For each $1\leqslant i\leqslant \Xi$ there exists a small ball\footnote{The balls in $\partial \mathcal{B}$ are the intersection with $\partial \mathcal{B}$ of the usual balls in the space $\mathbb{R}^2$.} $B_i$ in $\partial \mathcal{B}$ about $P_i$ such that for any $P\in B_i$
\begin{equation}
\mathfrak{A}_{\vec{n}(P_i), \vec{n}(P)} \asymp (K_{\vec{n}(P)})^{\frac{\omega_i-1}{\omega_i-2}}, \label{geometry2}
\end{equation}
where the implicit constants depend only on $\mathcal{B}$.
\end{lemma}

A consequence of this lemma is the following result which will be needed in \S \ref{non-vanishing}. This result is proved for $\xi \in S^1$, however, it can be easily extended to $\mathbb{R}^2_*$ since $K_{\xi}$ is positively homogeneous of degree zero.

\begin{lemma}\label{curv-in-ball} There exists a constant $c_1>0$ such that for any $\xi \in S^1$ with $K_{\xi}>0$
\begin{equation*}
K_\eta \asymp K_\xi    \qquad \textrm{if $\eta\in B(\xi, c_1(K_\xi)^{3/2})$}.
\end{equation*}
The constant $c_1$ and implicit constants depend only on $\mathcal{B}$.
\end{lemma}

\begin{proof}[Proof of Lemma \ref{curv-in-ball}] It suffices to prove this result for $\xi\in S^1$ such that $x(\xi)$ is in a small neighborhood in $\partial \mathcal{B}$ about a boundary point with curvature zero, say, $P_i$. Otherwise it follows easily from the mean value theorem. If $\eta\in B(\xi, c_1(K_\xi)^{3/2})$ then
\begin{equation*}
\mathfrak{A}_{\eta, \xi}\leqslant \frac{\pi}{2}\sin(\mathfrak{A}_{\eta, \xi})\leqslant \frac{\pi}{2}c_1 (K_\xi)^{\frac{\omega_i-1}{\omega_i-2}}.
\end{equation*}
To get the last inequality, we use $K_{\xi}<1$ and $\omega_i\geqslant 4$. But \eqref{geometry2} implies
\begin{equation*}
\mathfrak{A}_{\xi, \vec{n}(P_i)} \asymp (K_{\xi})^{\frac{\omega_i-1}{\omega_i-2}}.
\end{equation*}
Hence if $c_1$ is sufficiently small then $\mathfrak{A}_{\eta, \xi}\leqslant \mathfrak{A}_{\xi, \vec{n}(P_i)}/2$, which implies $1/2\leqslant \mathfrak{A}_{\eta, \vec{n}(P_i)}/\mathfrak{A}_{\xi, \vec{n}(P_i)}\leqslant 3/2$. By \eqref{geometry2} again, we get $K_\eta/K_{\xi} \asymp 1$.
\end{proof}

\section{Nonvanishing $2\times 2$ Determinants}\label{non-vanishing}

In this section we will give lower bounds of determinants of certain $2\times 2$ matrices (see Lemma \ref{non-vanishinglemma} below). This result is a refinement to M\"{u}ller~\cite{mullerII} Lemma 3 in two dimensional case with more precise bounds given. It is obtained based on M\"{u}ller's original proof.

The support function of $\mathcal{B}$ is given by $H(\xi)=\sup_{y\in \mathcal{B}}\langle \xi, y\rangle$ for any nonzero $\xi\in \mathbb{R}^2$. Then $H(\xi)=\langle \xi, x(\xi)\rangle$. It is positively homogeneous of degree one, {\it i.e.} $H(\lambda \xi)=\lambda H(\xi)$ if $\lambda>0$. Denote
\begin{equation*}
\mathcal{H}=\mathbb{R}^2\setminus\{r\vec{n}(P_i):\forall \, r\geqslant 0, i=1, \ldots, \Xi\}.
\end{equation*}

\begin{lemma}\label{upperforH} $H$ is smooth in $\mathcal{H}$ and for every $\xi\in \mathcal{H}$
\begin{equation*}
H(\xi)\lesssim |\xi|,
\end{equation*}
\begin{equation*}
D^{\nu}H(\xi) \lesssim 1 \quad \textrm{for $|\nu|=1$},
\end{equation*}
and
\begin{equation*}
D^{\nu}H(\xi) \lesssim |\xi|^{1-|\nu |}(K_\xi)^{3-2|\nu |} \quad \textrm{for } |\nu |\geqslant 2.
\end{equation*}
All implicit constants may depend only on $|\nu|$ and $\mathcal{B}$.
\end{lemma}

\begin{proof}[Proof of Lemma \ref{upperforH}] Assume $\vec{r}(s)=x_1(s)e_1+x_2(s)e_2$ is a parametrization of $\partial \mathcal{B}$ by arc length $s$. For every $\xi\neq 0$ there exists a unique $s(\xi)$ such that $x(\xi)=x_1(s(\xi))e_1+ x_2(s(\xi))e_2$, which leads to
\begin{equation*}
H(\xi)=\xi_1 x_1(s(\xi))+\xi_2 x_2(s(\xi)).
\end{equation*}

Since
\begin{equation*}
\xi/|\xi|=x_2'(s(\xi))e_1-x_1'(s(\xi)))e_2,
\end{equation*}
we get $\xi_1 x_1'(s(\xi))+\xi_2 x_2'(s(\xi))=0$. Note that if $\xi\in \mathcal{H}$ then
\begin{equation*}
K_\xi=x_1'(s(\xi))x_2''(s(\xi))-x_1''(s(\xi))x_2'(s(\xi))\neq 0.
\end{equation*}
It follows from the implicit function theorem that $s=s(\xi)$ is smooth at $\xi$. Hence $H$ is smooth in $\mathcal{H}$.

We can now estimate derivatives of $H$ by the implicit differentiation. Assume $\xi\in S^1\cap \mathcal{H}$. Differentiating $\xi_1 x_1'(s(\xi))+\xi_2 x_2'(s(\xi))=0$ yields
\begin{equation*}
\frac{\partial s}{\partial \xi_i}(\xi)=\frac{x_i'(s(\xi))}{x_1'(s(\xi))x_2''(s(\xi))-x_1''(s(\xi))x_2'(s(\xi))} \quad (i=1,2).
\end{equation*}
Continuing differentiating these two formulas, we get, by induction,
\begin{equation}
D_{\xi}^{\nu}s(\xi)\lesssim (K_\xi)^{1-2|\nu|} \quad \textrm{for $|\nu|\geqslant 1$}, \label{bound-s}
\end{equation}
where the implicit constant depends only on $|\nu|$ and $\mathcal{B}$. Differentiating $H$ gives
\begin{equation*}
\frac{\partial H}{\partial \xi_i}=x_i(s(\xi)) \lesssim 1 \quad (i=1,2).
\end{equation*}
Hence the bounds for $H$ follow from \eqref{bound-s} and the homogeneity of $H$.
\end{proof}

\begin{remark}\label{remark1}
The support function of $\mathcal{B}_{\theta}$ is given by $H_{\theta}(\xi)=\sup_{y\in \mathcal{B}_{\theta}}\langle \xi, y\rangle$. Denote $\mathcal{H}_{\theta}=R_{\theta}\mathcal{H}$. Since $H_{\theta}(\xi)=H(R_{\theta}^{t}\xi)$, we can easily get bounds for $H_{\theta}$ in the same form as in Lemma \ref{upperforH} (with $\mathcal{H}$ and $K_\xi$ replaced by $\mathcal{H}_{\theta}$ and $K_\xi^{\theta}$).
\end{remark}

The following result is concerning the Hessian matrix of $H$ and will be needed in the proof of next lemma. The proof is easy and we omit it.
\begin{lemma}\label{lemma4:2}
For any $\xi\in \mathcal{H}$, the matrix $\nabla^2_{\xi \xi}H(\xi)$ has two eigenvalues $0$ and $(|\xi|K_{\xi})^{-1}$.
\end{lemma}

Given vectors $v_1$, $v_2\in \mathbb{R}^2$, by writing $V=(v_1, v_2)$ we mean $V$ is the matrix with column vectors $v_1, v_2$. If $y\neq 0$ define $F_{\theta}(u_1, u_2)=H_{\theta}(y+u_1 v_1+u_2 v_2)$, $u_1, u_2\in \mathbb{R}$. For $q\in \mathbb{N}$ let
\begin{equation*}
h_q^{\theta}(y, v_1, v_2)=\det\left(g_{i, j} \right)_{1\leqslant i, j\leqslant 2},
\end{equation*}
where
\begin{equation*}
g_{i, j}=\frac{\partial^{q+2}F_{\theta}}{\partial u_1 \partial u_i \partial u_j \partial u_2^{q-1}}(0).
\end{equation*}

The main estimate in this section is the following lemma--the key preliminary for our (later) application of the method of stationary phase with nondegenerate critical points. This result is proved for $\xi\in S^1 \cap \mathcal{H}_{\theta}$, but can be easily extended to $\mathcal{H}_{\theta}$ due to the homogeneity of  $H_{\theta}$.

\begin{lemma} \label{non-vanishinglemma}
For every $\xi\in S^1 \cap \mathcal{H}_{\theta}$, there exist two orthogonal vectors $v_1^*(\xi)$, $v_2^*(\xi)\in \mathbb{Z}^2$ such that
\begin{equation}
|v_1^*|=|v_2^*|\asymp (K_{\xi}^{\theta})^{-4q}\quad \textrm{and}\quad \|(v_1^*, v_2^*)^{-1}\|\lesssim (K_{\xi}^{\theta})^{4q}, \label{lemma4:4-1}
\end{equation}
and a constant $c_2>0$ (depending only on $q$ and $\mathcal{B}$) such that for any $\eta \in B(\xi, c_2(K_\xi^{\theta})^{4q+2})$
\begin{equation}
|h_q^{\theta}(\eta, v_1^*, v_2^*)|\gtrsim (K_\xi^{\theta})^{-8q^2-16q-2},
\label{bbb}
\end{equation}
\begin{equation}
D^{\nu}H_{\theta}(\eta)\lesssim 1  \quad \textrm{for $0\leqslant |\nu|\leqslant 1$},\label{lemma4:4-3}
\end{equation}
and
\begin{equation}
D^{\nu}H_{\theta}(\eta) \lesssim (K_\xi^{\theta})^{3-2|\nu |} \quad \textrm{for $|\nu|\geqslant 2$}.\label{lemma4:4-4}
\end{equation}

The constants implicit in \eqref{lemma4:4-1} and \eqref{bbb} depend only on $q$ and $\mathcal{B}$. Those implicit in \eqref{lemma4:4-3} and \eqref{lemma4:4-4} depend only on $|\nu|$ and $\mathcal{B}$.
\end{lemma}

\begin{proof}[Proof of Lemma \ref{non-vanishinglemma}] We will essentially follow the proof of Lemma 3 in M\"{u}ller \cite{mullerII} (with some minor modification) and establish these inequalities through four steps for an arbitrarily fixed $\xi=(\xi_1, \xi_2)^{t}\in S^1 \cap \mathcal{H}_{\theta}$.

{\bf Step 1.} Denote $v_1=(-\xi_2, \xi_1)^{t}$ and $v_2=(\xi_1, \xi_2)^{t}$. We will first prove
\begin{equation}
h_q^{\theta}(\xi, v_1, v_2)=-q!^2 (K_{\xi}^{\theta})^{-2}.  \label{ddd}
\end{equation}

For $y=(y_1, y_2)^{t}$, set $\widetilde{H}_{\theta}(y)=H_{\theta}(My)$ where $M=(v_2, v_1)$ is an orthogonal matrix. Since $H_{\theta}$ is smooth at $\xi$, so is $\widetilde{H}_{\theta}$ at $e_1$. The Hessian matrix of $\widetilde{H}_{\theta}$ is
\begin{equation*}
\nabla^2\widetilde{H}_{\theta}(y)=M^{t} \nabla^2 H_{\theta}(My) M.
\end{equation*}
Since $\nabla^2 H_{\theta}(\xi)$ has two eigenvalues $0$ and $(K_{\xi}^{\theta})^{-1}$ by Lemma \ref{lemma4:2}, so does  $\nabla^2 \widetilde{H}_{\theta}(e_1)$. Note that
\begin{equation*}
H_{\theta}(\xi+u_1 v_1+u_2 v_2)=\widetilde{H}_{\theta}(e_1+u_1e_2+u_2e_1).
\end{equation*}
We use the latter expression to compute $h_q^{\theta}(\xi, v_1, v_2)$ since the following two equalities (derived from the homogeneity of $\widetilde{H}_{\theta}$; see the proof of M\"{u}ller~\cite{mullerII} Lemma 3) can simplify the computation:
\begin{equation}
(\widetilde{H}_{\theta})_{1j}(e_1)=(\widetilde{H}_{\theta})_{j1}(e_1)=0  \quad (1\leqslant j\leqslant 2); \label{keylemma1}
\end{equation}
\begin{equation}
\frac{\partial^{q+2}\widetilde{H}_{\theta}}{\partial y_1^q \partial y_i \partial y_j}(e_1)=(-1)^q q! (\widetilde{H}_{\theta})_{ij}(e_1) \quad (1\leqslant i,j\leqslant 2). \label{keylemma2}
\end{equation}

The equality \eqref{keylemma1} implies that $(\widetilde{H}_{\theta})_{22}(e_1)=(K_\xi^{\theta})^{-1}$. This, combined  with \eqref{keylemma2}, implies
\begin{equation*}
\frac{\partial^{q+2}}{\partial u_1 \partial u_i\partial u_2 \partial u_2^{q-1}}(H_{\theta}(y+u_1 v_1+u_2 v_2))(0)=\delta_{1i}(-1)^q q!(K_{\xi}^{\theta})^{-1},
\end{equation*}
where $\delta_{ij}$ is the Kronecker notation. This equality easily leads to \eqref{ddd}.

{\bf Step 2.} For any $N\in \mathbb{N}$ there exist two integers $N_l$  ($l=1, 2$) such that
$|\xi_l-N_l/N|\leqslant 1/N$. Denote $\widetilde{v}_1=(-N_2/N, N_1/N)^{t}$ and $\widetilde{v}_2=(N_1/N, N_2/N)^{t}$. Then $|v_l-\widetilde{v}_l|\leqslant \sqrt{2}/N$. If $N\geqslant 2\sqrt{2}$ then $1/2 \leqslant|\widetilde{v}_1|=|\widetilde{v}_2|\leqslant 3/2$. By the mean value theorem and Lemma \ref{upperforH} we get
\begin{equation*}
|h_q^{\theta}(\xi, \widetilde{v}_1, \widetilde{v}_2)-h_q^{\theta}(\xi, v_1, v_2)|\leqslant C_1N^{-1} (K_\xi^{\theta})^{-4q-2},
\end{equation*}
where $C_1$ depends only on $q$ and $\mathcal{B}$. Let $N$ be the smallest integer not less than $2C_1q!^{-2}(K_\xi^{\theta})^{-4q}$. Then
\begin{equation*}
|h_q^{\theta}(\xi, \widetilde{v}_1, \widetilde{v}_2)|\geqslant q!^2(K_\xi^{\theta})^{-2}/2.
\end{equation*}

{\bf Step 3.} Set $v_1^*=N \widetilde{v}_1$ and $v_2^*=N \widetilde{v}_2$. Then $v_1^*$ and $v_2^*$ are two orthogonal integral vectors such that $|v_1^*|=|v_2^*|\asymp_{q, \mathcal{B}} (K_{\xi}^{\theta})^{-4q}$ and
\begin{equation*}
|h_q^{\theta}(\xi, v_1^*, v_2^*)|=N^{2q+4}|h_q^{\theta}(\xi, \widetilde{v}_1, \widetilde{v}_2)| \gtrsim_{q, \mathcal{B}}  (K_\xi^{\theta})^{-8q^2-16q-2}.
\end{equation*}

Since $(v_1^*, v_2^*)=N(\tilde{v}_1, \tilde{v}_2)$ its inverse matrix is
\begin{equation*}
(v_1^*, v_2^*)^{-1}=N^{-1}(\textrm{adjugate matrix of $(\tilde{v}_1, \tilde{v}_2)$})/\det(\tilde{v}_1, \tilde{v}_2),
\end{equation*}
followed by $\|(v_1^*, v_2^*)^{-1}\|\lesssim_{q, \mathcal{B}} (K_{\xi}^{\theta})^{4q}$.

{\bf Step 4.} Assume $\eta\in B(\xi, c_2(K_\xi^{\theta})^{4q+2})$ with $c_2$ chosen below. If $c_2$ is sufficiently small, Lemma \ref{curv-in-ball} implies $K_\eta^{\theta} \asymp K_\xi^{\theta}$. Recalling also Remark \ref{remark1} we immediately get \eqref{lemma4:4-3} and \eqref{lemma4:4-4}.

By the mean value theorem and the assumption $|\eta-\xi|\leqslant c_2(K_\xi^{\theta})^{4q+2}$ we get
\begin{equation*}
|h_q^{\theta}(\eta, v_1^*, v_2^*)-h_q^{\theta}(\xi, v_1^*, v_2^*)|\leqslant  C_2 c_2 (K_\xi^{\theta})^{-8q^2-16q-2}.
\end{equation*}
where $C_2$ depends only on $q$ and $\mathcal{B}$. The inequality \eqref{bbb} follows if $c_2$ is sufficiently small. This finishes the proof.
\end{proof}


\section{The Fourier Transform of Certain Indicator Functions}\label{four-tran}

If $\mathcal{B}$ is a compact convex planar domain with smooth boundary and \emph{positive} Gaussian curvature.
H\"{o}rmander~\cite{hormander} Corollary 7.7.15 gives the following asymptotic formula for the Fourier transform of the indicator function $\chi_{\mathcal{B}}$ for $\lambda>1$ and \emph{every} $\xi\in S^{d-1}$
\begin{equation*}
\widehat{\chi}_{\mathcal{B}}(\lambda \xi)=[A_1K_{\xi}^{-1/2}e^{-2\pi i\lambda \langle \xi, x(\xi) \rangle}+A_2K_{-\xi}^{-1/2}e^{2\pi i\lambda \langle -\xi, x(-\xi) \rangle}]\lambda^{-3/2}+O(\lambda^{-5/2}),
\end{equation*}
where $A_1=e^{3\pi i/4}/2\pi$, $A_2=e^{-3\pi i/4}/2\pi$. This formula is not good for us since the domains we consider are of finite type.

Randol~\cite{randol_R^2} studied the Fourier transforms of the indicator function of a compact (not necessarily convex) planar domain $\mathcal{B}$ of finite type. In particular, his Theorem 1 gave an upper bound for
\begin{equation*}
\Phi(\xi)=\sup\limits_{r>0} r^{3/2}|\widehat{\chi}_{\mathcal{B}}(r\xi)|, \qquad \textrm{$\xi\in S^1$},
\end{equation*}
which blows up as the curvature goes to zero. It says that $\Phi(\xi)$ is always bounded, except in neighborhoods of those points of $S^1$ corresponding to exterior or interior normals to $\partial \mathcal{B}$ at points with curvature zero. In a neighborhood of such a point $\xi_0$,
\begin{equation}
\Phi(\xi)\lesssim (\mathfrak{A}_{\xi, \xi_0})^{-\frac{\omega_0-2}{2(\omega_0-1)}}, \label{randolbound}
\end{equation}
where $\omega_0$ is the largest type at those points of $\partial \mathcal{B}$ at which the exterior normal is either $\xi_0$ or $-\xi_0$. For convex domains of finite type, this bound follows easily from Lemma \ref{lemma2:1} and the argument on \cite[P.~19]{svensson}.

For our purpose we need an asymptotic formula for $\widehat{\chi}_{\mathcal{B}}(\xi)$. We first prove the following lemma.
\begin{lemma}\label{monotonicity} Assume $\mathcal{B}$ is a compact strictly convex planar domain with smooth boundary. Then there exist two positive constants $c$ and $c_3$ (both depending only on $\mathcal{B}$) such that, for any $\xi \in S^1$ and $r\leqslant c_3$,
\begin{equation}
|\langle \vec{n}(x), \xi\rangle|\leqslant 1-c r^2(\min(K_{\xi}, K_{-\xi}))^4 \label{ineq-mono}
\end{equation}
if $x$ is in $\partial \mathcal{B}\setminus (B(x(\xi), r K_{\xi})\cup B(x(-\xi), r K_{-\xi}))$.
\end{lemma}

\begin{proof}[Proof of Lemma \ref{monotonicity}]
Note that there exists a $C_0>0$ such that, for any $\xi\in S^1$, the boundary $\partial \mathcal{B}$ in a neighborhood of $x(\xi)$ can be parametrized by
\begin{equation}
\vec{r}(u)=x(\xi)+u \vec{t}(x(\xi))+h(u, \xi)(-\xi) \qquad u\in I=[-C_0, C_0], \label{parametrization}
\end{equation}
where $h(\cdot \, , \xi)\in C^{\infty}(I)$ for all $\xi\in S^1$ such that $h(0, \xi)=0$. Note that \eqref{parametrization} implies $h'_u(0, \xi)=0$ and $h''_u(0, \xi)=K_{\xi}$. Since the map
\begin{equation*}
\xi\in S^1 \mapsto h(\cdot \, , \xi)
\end{equation*}
is continuous and its domain is compact we have
\begin{equation}
|\partial^{j}_{u}h(u, \xi)|\leqslant C_1 \quad \textrm{for any $\xi \in S^1$, $j\in \mathbb{N}_0$, and $u\in I$}. \label{uniformbound}
\end{equation}

Denote by $K_{max}$ the largest curvature of $\partial \mathcal{B}$. Let
\begin{equation*}
c_3=\min (\frac{C_0}{K_{max}}, \frac{1}{C_1}, \frac{2\sqrt{6}}{9 K_{max}^2})
\end{equation*}
and $r\leqslant c_3$. Due to the symmetry and monotonicity it suffices to prove \eqref{ineq-mono} only for $x\in \partial \mathcal{B}\cap \partial B(x(\xi), r K_{\xi})$. Since $r\leqslant C_0/K_{max}$ there exists a $u_*\in I$ such that $x=\vec{r}(u_*)$ with $rK_{\xi}(1+C_1^2)^{-1/2}\leqslant |u_*|\leqslant rK_{\xi}$.

Since $r\leqslant 1/C_1$, Taylor's formula and the size of $u_*$ yields
\begin{equation}
r K_{\xi}^2(1+C_1^2)^{-1/2}/2\leqslant u_*K_{\xi}/2 \leqslant |h'_u(u_*, \xi)|\leqslant 3u_*K_{\xi}/2\leqslant 3 r K_{\xi}^2/2. \label{bound3}
\end{equation}
By Taylor's formula again,
\begin{equation*}
\langle \vec{n}(x), \xi\rangle=(1+h'_u(u_*, \xi)^2)^{-1/2}=1-h'_u(u_*, \xi)^2/2+\textrm{R}
\end{equation*}
with a remainder $|\textrm{R}|\leqslant 3h'_u(u_*, \xi)^4/8\leqslant h'_u(u_*, \xi)^2/4$. The last inequality follows from \eqref{bound3} and $r\leqslant 2\sqrt{6}/9 K_{max}^2$. By \eqref{bound3} again we get
\begin{equation}
\langle \vec{n}(x), \xi\rangle \leqslant 1-c r^2 K_{\xi}^4, \notag
\end{equation}
where $c=1/(16+16C_1^2)$.
\end{proof}

\begin{theorem}\label{asym-R^2} Let $\mathcal{B}$ be a compact strictly convex planar domain with smooth boundary, $s$ the arc length on $\partial \mathcal{B}$, and $n_l$ ($l=1, 2$) the $l^{\textrm{th}}$ component of the Gauss map of $\partial \mathcal{B}$. For $\xi\in S^1$ with $\delta_{\xi}=\min(K_{\xi}, K_{-\xi})>0$, we have
\begin{equation*}
\begin{split}
\widehat{n_l ds}(\lambda \xi)&=\big(e^{\pi i/4}\xi_l(K_{\xi})^{-1/2}e^{-2\pi i\lambda H(\xi)}\\
&\quad +e^{-\pi i/4}(-\xi_l)(K_{-\xi})^{-1/2}e^{2\pi i\lambda H(-\xi)} \big)\lambda^{-1/2}\\
&\quad\quad +O(\lambda^{-3/2}(\delta_{\xi})^{-7/2}+\lambda^{-N}(\delta_{\xi})^{-4N})   \quad \textrm{for $\lambda >0$},
\end{split}
\end{equation*}
where $N\in \mathbb{N}$ and the implicit constant depends only on $N$ and $\mathcal{B}$.
\end{theorem}

\begin{proof}[Proof of Theorem \ref{asym-R^2}] As in the proof of Lemma \ref{monotonicity}, the boundary $\partial \mathcal{B}$ in a neighborhood of $x(\xi)$ can be parametrized by \eqref{parametrization} with a uniform upper bound as in \eqref{uniformbound} and we assume $C_0$, $C_1$, $c_3$, and $K_{max}$ are constants appearing there. Let
\begin{equation*}
c_4=\min(\frac{C_0}{K_{max}}, \frac{3}{2C_1(1+C_0)}, 2c_3).
\end{equation*}
Decompose $n_l$ as a sum $n_l=\psi_1+\psi_2+\psi_3$ where
\begin{equation*}
\psi_1(x, \xi)=n_l(x)\chi_0(\frac{x-x(\xi)}{c_4 K_{\xi}}) \textrm{ and } \psi_2(x, \xi)=n_l(x)\chi_0(\frac{x-x(-\xi)}{c_4 K_{-\xi}}).
\end{equation*}

We first estimate $\widehat{\psi_1 ds}$ and by \eqref{parametrization}
\begin{equation}
\widehat{\psi_1 ds}(\lambda \xi)=e^{-2\pi i\lambda \langle \xi, x(\xi) \rangle} \int \tau(u, \xi) e^{2\pi i\lambda h(u, \xi)} \,du, \label{1stpart}
\end{equation}
where $\tau(u, \xi)=\psi_1(\vec{r}(u), \xi)(1+h'_u(u, \xi)^2)^{1/2}$ such that
\begin{equation*}
\tau(\cdot \, , \xi)\in  C^{\infty}_{c}(-c_4 K_{\xi}, c_4 K_{\xi})
\end{equation*}
and
\begin{equation*}
|\partial_{u}^{j}\tau(u, \xi) |\leqslant C(\chi_0, C_1) (c_4 K_{\xi})^{-j}.
\end{equation*}

Denote the integral in \eqref{1stpart} by $\Delta(\xi)$. By a change of variable,
\begin{equation*}
\Delta(\xi)=K_{\xi} \int \tau(K_{\xi}u, \xi) e^{2\pi i\lambda h(K_{\xi}u, \xi)} \,du.
\end{equation*}
By Taylor's formula,
\begin{equation*}
h(K_{\xi}u, \xi)=(K_{\xi})^3 u^2(1+\varepsilon(u, \xi))/2, \qquad  u\in [-c_4, c_4],
\end{equation*}
where $\varepsilon(u, \xi)=u\int_0^1 \partial_u^3 h(K_{\xi}um, \xi) (1-m)^2 \, dm$. Since $1/2 \leqslant 1+\varepsilon(u, \xi)\leqslant 3/2$ (due to $c_4\leqslant 3/(2C_1)$), we can define $v=u(1+\varepsilon(u, \xi))^{1/2}$. Since $\partial_u v(u, \xi)>\sqrt{2}/4$ (due to $c_4\leqslant \min(C_0/K_{max}, 3/2C_1(1+C_0))$), then $u\mapsto v$ is a smooth invertible mapping from $(-c_4, c_4)$ to a neighborhood of $0$ in $v$-space such that $|\partial_u^j v|\leqslant C(C_0, C_1)$, $|\partial_v^j u|\leqslant C(C_0, C_1)$, and
\begin{equation*}
h(K_{\xi}u, \xi)=(K_{\xi})^3 v^2/2.
\end{equation*}
Then
\begin{equation*}
\Delta(\xi)=K_{\xi} \int \tilde{\tau}(v, \xi) e^{i\tilde{\lambda} v^2/2} \,dv,
\end{equation*}
where $\tilde{\lambda}=2\pi (K_{\xi})^3\lambda$ and $\tilde{\tau}(v, \xi)=\tau(K_{\xi}u(v), \xi)\partial_v u$. Applying Lemma \ref{app:lemma:4} (with $k=1$ there) to the integral above yields an asymptotic expansion, which in turn gives
\begin{equation*}
\widehat{\psi_1 ds}(\lambda \xi)=e^{\pi i/4} \xi_l(K_{\xi})^{-1/2}e^{-2\pi i\lambda \langle \xi, x(\xi) \rangle}\lambda^{-1/2} +O(\lambda^{-3/2}(K_{\xi})^{-7/2}),
\end{equation*}
where the implicit constant depends only on $\mathcal{B}$.

Since $\widehat{\psi_2 ds}$ is similar, it remains to estimate $\widehat{\psi_3 ds}$. Assume $\vec{r}: s\in [0, L]\mapsto \vec{r}(s)\in \partial \mathcal{B}$ is a parametrization of $\partial \mathcal{B}$ by arc length and $\vec{r}(0)=x(\xi)$. Then
\begin{equation*}
\widehat{\psi_3 ds}(\lambda \xi)=\int \tau_1(s, \xi) e^{-2\pi i\lambda f(s,\xi)} \,ds,
\end{equation*}
where $\tau_1(s, \xi)=\psi_3(\vec{r}(s), \xi)$ and $f(s,\xi)=\langle \vec{r}(s), \xi \rangle$. Note that
\begin{equation*}
f'_s(s, \xi)=\langle \vec{t}(\vec{r}(s)), \xi \rangle.
\end{equation*}
But $|\langle \vec{t}(\vec{r}(s)), \xi\rangle|^2+|\langle \vec{n}(\vec{r}(s)), \xi\rangle|^2=1$ and Lemma \ref{monotonicity} ($c_4\leqslant 2c_3$) yields, for any $s$ such that $\tau_1(s, \xi)\neq 0$, that
\begin{equation*}
|\langle \vec{n}(\vec{r}(s)), \xi\rangle|\leqslant 1-c c_4^2(\delta_{\xi})^4/4.
\end{equation*}
Hence $|\langle \vec{n}(\vec{r}(s)), \xi\rangle|^2\leqslant 1-c c_4^2(\delta_{\xi})^4/4$. It follows that
\begin{equation*}
|f'_s(s, \xi)|\geqslant \sqrt{c} c_4 (\delta_{\xi})^2/2.
\end{equation*}

Note that $\partial^{j}_s f \lesssim 1$ and $\partial^{j}_s \tau_1 \lesssim (\delta_{\xi})^{-j}$, thus by Lemma \ref{app:lemma:2} we get
\begin{equation*}
\widehat{\psi_3 ds}(\lambda \xi)\lesssim \lambda^{-N}(\delta_{\xi})^{-4N},
\end{equation*}
where the implicit constant depends only on $N$ and $\mathcal{B}$.

\end{proof}

As a consequence of the Gauss--Green formula we get:

\begin{corollary}\label{FTofInd} Let $\mathcal{B}$ be a compact strictly convex planar domain with smooth boundary. For $\xi\in S^1$ with $\delta_{\xi}^{\theta}=\min(K_{\xi}^{\theta}, K_{-\xi}^{\theta})>0$, we have
\begin{align*}
\widehat{\chi}_{\mathcal{B}_{\theta}}(\lambda \xi)&=\big((2\pi)^{-1}e^{3\pi i/4}(K_{\xi}^{\theta})^{-1/2}e^{-2\pi i\lambda H_{\theta}(\xi)}\\
  &\quad+(2\pi)^{-1}e^{-3\pi i/4}(K_{-\xi}^{\theta})^{-1/2}e^{2\pi i\lambda H_{\theta}(-\xi) } \big)\lambda^{-3/2}\\
  &\quad\quad+O(\lambda^{-5/2}(\delta_{\xi}^{\theta})^{-7/2}+\lambda^{-N-1}(\delta_{\xi}^{\theta})^{-4N}) \quad \textrm{for $\lambda >0$},
\end{align*}
where $N\in \mathbb{N}$ and the implicit constant depends only on $N$ and $\mathcal{B}$.
\end{corollary}

\begin{remark} In \S \ref{mainterm}, we will apply this result ($N=1$) to convex planar domains of finite type. The error term becomes $O(\lambda^{-2}(\delta_{\xi}^{\theta})^{-4})$.
\end{remark}

\begin{proof}[Proof of Corollary \ref{FTofInd}] This result follows easily from
\begin{equation*}
\widehat{\chi}_{\mathcal{B}_{\theta}}(\lambda \xi)=\widehat{\chi}_{\mathcal{B}}(\lambda R_{\theta}^{t}\xi),
\end{equation*}
\begin{equation*}
2\pi i\lambda \xi_{l}\widehat{\chi}_{\mathcal{B}}(\lambda \xi)=-\widehat{n_l ds}(\lambda \xi),
\end{equation*}
and Theorem \ref{asym-R^2}.
\end{proof}

\section{Estimate of Exponential Sums}\label{proof-claims}

Let $M_*>1$ and $T>0$ be parameters. In this section we consider two-dimensional exponential sums of the form
\begin{equation*}
S(T, M_*; G, F)=\sum_{m\in\mathbb{Z}^2} G(\frac{m}{M_*}) e(TF(\frac{m}{M_*})),
\end{equation*}
where $G:\mathbb{R}^2\rightarrow\mathbb{R}$ is $C^{\infty}$ smooth, compactly supported, and bounded above by a constant, and
$F:\Omega\subset \mathbb{R}^2 \rightarrow\mathbb{R}$ is $C^{\infty}$ smooth on an open convex domain $\Omega$ such that
\begin{equation*}
\textrm{supp} (G)\subset \Omega \subset c_0 B(0, 1),
\end{equation*}
where  $c_0>0$ is a fixed constant. Here we quote Lemma 2.2 in Guo~\cite{guo} ($d=2$), a variant of Lemma 1 in M\"{u}ller~\cite{mullerII}.

\begin{lemma}\label{weyl-van-inequality}
Let $q\in \mathbb{N}$, $Q=2^q$, and $r_1, \ldots, r_q\in \mathbb{Z}^2$ be nonzero integral vectors with $|r_i|\lesssim 1$. Furthermore, let $H$ be a real parameter which satisfies $1<H\lesssim M_*$. Set $H_l=H_{q,l}=H^{2^{l-q}}$ for $l=1, \ldots, q$. Then
\begin{equation*}
|S(T, M_*; G, F)|^{Q}\lesssim \frac{M_*^{2Q}}{H} + \frac{M_*^{2(Q-1)}}{H_1\cdot \ldots \cdot H_q} \sum_{\substack{1\leqslant h_i<H_i \\1 \leqslant i \leqslant
q}} |S(\mathscr{H}TM_*^{-q}, M_*; G_q, F_q)|,
\end{equation*}
where $\mathscr{H}=\prod_{l=1}^{q}h_l$ and functions $G_q$, $F_q$ are defined as follows:
\begin{equation*}
G_q(x)=G_q(x, h_1, \ldots, h_q)=\prod_{\substack{u_i\in\{0,1\}\\
1\leqslant i\leqslant q}} G(x+\sum_{l=1}^{q} \frac{h_l}{M_*} u_l r_l)
\end{equation*}
and
\begin{align*}
F_q(x)&=F_q(x, h_1, \ldots, h_q)\\
      &=\int_{(0,1)^q} \langle r_1, \nabla\rangle \cdots \langle r_q, \nabla\rangle F
       (x+\sum_{l=1}^{q} \frac{h_l}{M_*} u_l r_l) \, du_1\ldots du_q.
\end{align*}

The integral representation of $F_q$ is well defined on
the open convex set $\Omega_q=\Omega_q(h_1, \ldots,
h_q)=\{x\in \Omega : x+\sum_{l=1}^{q} (h_l/M_*) u_l r_l \in
\Omega \textrm{ for all } u_l\in \{0,1\}, l=1, \ldots, q\}$. And supp$(G_q)\subset \Omega_q\subset \Omega$.
\end{lemma}


\begin{proposition}\label{claim1} Let $q\in \mathbb{N}$, $Q=2^q$, and $K<1$ be a positive parameter. Assume that
\begin{equation}
\dist(\supp(G), \Omega^{c})\gtrsim K^{4q+2}, \label{bdy-dist}
\end{equation}
and that for all $\nu\in \mathbb{N}_0^2$ and $y\in \Omega$
\begin{equation}
D^{\nu}G(y)\lesssim K^{-(4q+2)|\nu|}, \label{aaa}
\end{equation}
\begin{equation}
D^{\nu}F(y)\lesssim
\bigg\{ \begin{array}{ll}
1 & \textrm{if $0\leqslant |\nu|\leqslant 1$}\\
K^{3-2|\nu|} & \textrm{if $|\nu|\geqslant 2$}
\end{array},  \label{upperforf}
\end{equation}
and for $\mu=(1, q-1)$
\begin{equation}
\big|\det( \nabla^2 D^{\mu}F(y) )\big|\gtrsim K^{-2}.   \label{lowerforf}
\end{equation}

If $M_*\geqslant K^{-4q-2}$ and $T$ is restricted to
\begin{equation}
T\geqslant K^{(8q+4)/Q-5}M_*^{q-1+2/Q}, \label{restriction}
\end{equation}
then
\begin{equation}
S(T, M_*; G, F)\lesssim (K^{-12q-1}TM_*^{6Q-q-6})^{1/(3Q-2)}+R, \label{kkk}
\end{equation}
where
\begin{equation}
\begin{split}
R=&K^{-(20q+4)/Q}M_*^{2-2/Q}[K^{-(12q+4)/Q}+\\
  &(K^{12q+1}T^{-1}M_*^{q+2})^{1/(3Q-2)} (  (K^{-20q-7}T^{-1}M_*^{q})^{1/Q}+(\log M_*)^{1/Q})]
\end{split}\label{longremainder}
\end{equation}

The constant implicit in \eqref{kkk} depends only on $q$, $c_0$, and constants implicit in \eqref{bdy-dist}, \eqref{aaa}, \eqref{upperforf}, and \eqref{lowerforf}.
\end{proposition}

\begin{remark} This result is similar with Theorem 2 in \cite{mullerII} and Proposition 2.4, 2.5 in \cite{guo}, but here there is an extra parameter $K$ in various bounds.
\end{remark}

\begin{proof}[Proof of Proposition \ref{claim1}] Let $1<H\leqslant c_5 K^{4q+2}M_*$ with $c_5<1$ chosen (later) to be sufficiently small. Then $H\leqslant M_*$.  We use Lemma \ref{weyl-van-inequality} with $r_1=e_1$ and $r_j=e_2$ ($j=2, \ldots, q$). Applying to $S_4:=S(\mathscr{H}TM_*^{-q}, M_*; G_q, F_q)$ the Poisson summation formula followed by a change of variables $x=K^2 M_* z$ yields
\begin{align*}
S_4 &=\sum_{p\in \mathbb{Z}^2}\int_{\mathbb{R}^2} G_q(x/M_*)e(\mathscr{H}TM_*^{-q}F_q(x/M_*)-\langle p, x\rangle) \,dx\\
    &=\sum_{p\in \mathbb{Z}^2} K^{4}M_*^2\int_{\mathbb{R}^2} \Psi_q(z)e(\mathscr{H}TM_*^{-q}F_q(K^2 z)-K^{2}M_* \langle p, z\rangle)\,dz,
\end{align*}
where $\Psi_q(z)=G_q(K^{2}z)$. It is obvious that
\begin{equation}
\supp(\Psi_q)\subset K^{-2}\Omega_q \subset c_0 K^{-2}B(0, 1). \label{domain1}
\end{equation}
By \eqref{bdy-dist} we also have
\begin{equation}
\dist(\supp(\Psi_q), (K^{-2}\Omega_q)^{c})\gtrsim K^{4q}. \label{domain2}
\end{equation}

By the assumption \eqref{upperforf} there exists a constant $A_1$ such that
\begin{equation*}
|\nabla_z(F_q(K^{2}z))|\leqslant (A_1/2) K^{3-2q}.
\end{equation*}
We split $S_4$ into two parts
\begin{equation*}
S_4=\sum_{|p|<A_1K^{1-2q}\mathscr{H}TM_*^{-q-1}}+\sum_{|p|\geqslant A_1K^{1-2q}\mathscr{H}TM_*^{-q-1}}=\textrm{:}S_5+R_5.
\end{equation*}

We will prove the following lemma (later) by integration by parts:
\begin{lemma}\label{R5}
\begin{equation*}
R_5\lesssim K^{-12q-6}M_*^{-1}.
\end{equation*}
\end{lemma}

Next we will estimate $S_5$. Define $\lambda_1=K^{3-2q}\mathscr{H}TM_*^{-q}$ and
\begin{equation*}
\Phi_q(z, p)=(\mathscr{H}TM_*^{-q}F_q(K^2 z)-K^{2}M_* \langle p, z\rangle)/\lambda_1,
\end{equation*}
then
\begin{equation}
S_5=K^{4}M_*^2 \sum_{|p|<A_1K^{1-2q}\mathscr{H}TM_*^{-q-1}} \int_{\mathbb{R}^2} \Psi_q(z)e(\lambda_1 \Phi_q(z, p))\,dz. \label{S5}
\end{equation}

For all $z\in K^{-2}\Omega_q$, by \eqref{aaa}, \eqref{upperforf}, and \eqref{lowerforf},
\begin{equation}
D^{\nu}_z\Psi_q(z)\lesssim K^{-4q|\nu|}, \label{ff}
\end{equation}
\begin{equation}
D^{\nu}_z \Phi_q(z, p)\lesssim \bigg\{
\begin{array}{ll}
K^{-2} & \textrm{for $\nu=0$}\\
1 & \textrm{for $|\nu|\geqslant 1$}
\end{array},   \label{gg}
\end{equation}
and
\begin{equation}
|\det\big(\nabla^2_{zz}\Phi_q(z, p) \big)|\gtrsim K^{4q}. \label{cc}
\end{equation}

To prove this lower bound \eqref{cc} we first note, by using the definition of $F_q$ and the mean value theorem, that for $\mu=(1, q-1)$
\begin{equation*}
\frac{\partial^2}{\partial z_{l_1} \partial z_{l_2}} (\Phi_q(z, p))=K^{2q+1}\frac{\partial^{2} D^{\mu}F}{\partial x_{l_1} \partial x_{l_2}}(K^{2}z)+O(K^{-2}\frac{H}{M_*}).
\end{equation*}
The two terms on the right are $\lesssim$ $1$ and $c_5 K^{4q}$ respectively. Thus
\begin{equation*}
\det(\nabla^2_{zz}(\Phi_q(z, p)))=K^{4q+2}\det(\nabla^2D^{\mu}F(K^{2}z))+O(c_5 K^{4q}).
\end{equation*}
By \eqref{lowerforf}, we get \eqref{cc} if we pick a sufficiently small $c_5$.

With \eqref{domain1}, \eqref{domain2}, \eqref{ff}, \eqref{gg}, and \eqref{cc}, we can estimate the integrals in sum $S_5$. Let us fix an arbitrary $|p|<A_1K^{1-2q}\mathscr{H}TM_*^{-q-1}$. We first estimate the number of critical points of the phase function $\Phi_q$. Denote $\widetilde{p}=K^{2}M_*p/\lambda_1$ and $F(z)=K^{2q-3}\nabla_z(F_q(K^{2}z))$, then $\nabla_z \Phi_q(z, p)=F(z)-\widetilde{p}$ and critical points are determined by the equation
\begin{equation*}
F(z)=\widetilde{p} \quad \textrm{for $z\in K^{-2}\Omega_q$}.
\end{equation*}

The bounds \eqref{gg} and \eqref{cc} imply that the mapping $F=(F_1, F_2)$ satisfies
\begin{equation*}
D^{\nu}F_j(z)\lesssim 1   \quad \textrm{for $|\nu|\leqslant 2$, $j=1$, $2$},
\end{equation*}
and
\begin{equation*}
|\det(\nabla_z F(z))|\gtrsim K^{4q}.
\end{equation*}

By \eqref{domain2}, we know that $\supp(\Psi_q)$ is strictly smaller than $K^{-2}\Omega_q$ and the distance between their boundary is larger than $a_1 K^{4q}$ for some positive constant $a_1$. Let $r_0=a_1 K^{4q}/2$. By Taylor's formula, there exists a positive constant $a_2$ ($<a_1/2$) such that if $\tilde{z}$ is a critical point in $(\supp(\Psi_q))_{(r_0)}$ then
\begin{equation}
|\nabla_z \Phi_q(z, p)|\gtrsim K^{4q}|z-\tilde{z}|, \quad \textrm{for any $z\in B(\tilde{z}, a_2 K^{4q})$}. \label{taylor}
\end{equation}

Applying Lemma \ref{app:lemma:1} to $F$ with $r_0$ as above yields two positive constants $a_3$ ($<a_2/2$) and $a_4$ such that if $r_1=a_3 K^{4q}$, $r_2=a_4 K^{8q}$, then $F$ is bijective from $B(z, 2r_1)$ to an open set containing $B(F(z), 2r_2)$ for any $z\in (\supp(\Psi_q))_{(r_0)}$. It follows, simply by a size estimate, that the number of critical points in $(\supp(\Psi_q))_{(r_1)}$ is $\lesssim$ $(K^{-2}/r_1)^{2}\lesssim K^{-8q-4}$ .

Let $Z_j$ ($j=1, \ldots, J(p)$) be all critical points in $(\supp(\Psi_q))_{(r_1)}$ corresponding to the $p$ we fixed. Let $\chi_j(z)=\chi_0((z-Z_j)/(c_6 r_1))$ with $c_6$ chosen (below) to be sufficiently small, then
\begin{equation}
\int \Psi_q(z)e(\lambda_1 \Phi_q(z, p))\,dz =\sum_{j=1}^{J(p)}\int \chi_j(z) \Psi_q(z)e(\lambda_1 \Phi_q(z, p))\,dz+R_6,\label{S5-1}
\end{equation}
where
\begin{equation*}
R_6=\int [1-\sum_{j=1}^{J(p)}\chi_j(z)] \Psi_q(z)e(\lambda_1 \Phi_q(z, p))\,dz.
\end{equation*}

For each $j=1, \ldots, J(p)$, we consider a new phase function $\phi_j(z, p)=\Phi_q(z, p)-\Phi_q(Z_j, p)$ satisfying $D_z^{\nu}\phi_j(z, p)\lesssim 1$. By Lemma \ref{app:lemma:3} (with $\lambda=\lambda_1$, $\delta=K^{4q}$), if $c_6$ is sufficiently small then
\begin{align}
\big|\int \chi_j(z) &\Psi_q(z)e(\lambda_1 \Phi_q(z, p))\,dz\big| \nonumber \\
  &=\big|\int \chi_j(z) \Psi_q(z)e(\lambda_1 \phi_j(z, p))\,dz\big|\lesssim \lambda_1^{-1} K^{-2q}.\label{S5-2}
\end{align}

We will prove (later), by integration by parts, that
\begin{lemma}\label{R6}
\begin{equation}
R_6\lesssim K^{-32q-8}\lambda_1^{-2}. \label{boundR6}
\end{equation}
\end{lemma}
Using \eqref{S5}, \eqref{S5-1}, \eqref{S5-2}, and \eqref{boundR6}, we get
\begin{align*}
S_5&\lesssim K^{4}M_*^2\big(1+(A_1K^{1-2q}\mathscr{H}TM_*^{-q-1})^2\big) \big(\frac{\lambda_1^{-1}K^{-2q}}{K^{8q+4}} +\frac{\lambda_1^{-2}}{K^{32q+8}}\big)\\
   &\lesssim K^{-12q-1}\mathscr{H}TM_*^{-q}+R_7,
\end{align*}
where
\begin{equation*}
R_7=K^{-8q-3}(\mathscr{H}T)^{-1}M_*^{q+2}+K^{-28q-10}(\mathscr{H}T)^{-2}M_*^{2q+2}
   +K^{-32q-8}.
\end{equation*}

Recall that $R_5\lesssim K^{-12q-6}M_*^{-1}$, hence $R_5\lesssim K^{-32q-8}$ and
\begin{equation*}
S_4=S_5+R_5\lesssim K^{-12q-1}\mathscr{H}TM_*^{-q}+R_7.
\end{equation*}

Plugging this bound into the inequality in Lemma \ref{weyl-van-inequality} gives
\begin{equation}
|S(T, M_*; G, F)|^{Q}\lesssim M_*^{2Q}H^{-1}+K^{-12q-1}TM_*^{2Q-q-2}H^{2-2/Q}+\textrm{E}, \label{cccc}
\end{equation}
where
\begin{equation*}
\begin{split}
\textrm{E}&=M_*^{2(Q-1)}(K^{-8q-3}T^{-1}M_*^{q+2}H^{-2+2/Q}\log H \\
  &\quad +K^{-32q-8}+K^{-28q-10}T^{-2}M_*^{2q+2}H^{-2+2/Q}).
\end{split}
\end{equation*}

In order to balance the first two terms on the right side of \eqref{cccc} we let
\begin{equation*}
H=c_5(K^{12q+1}T^{-1}M_*^{q+2})^{Q/(3Q-2)}.
\end{equation*}
The assumption \eqref{restriction} implies $H\leqslant c_5 K^{4q+2} M_*$. Since we can assume
\begin{equation*}
T<c_7 K^{12q+1}M_*^{q+2}
\end{equation*}
with a sufficiently small $c_7$ (otherwise the trivial bound of $S(T, M_*; G, F)$, {\it i.e.} $M_*^2$, is better than \eqref{kkk}), it implies $1<H$. With the choice of $H$ as above, \eqref{cccc} leads to \eqref{kkk}.
\end{proof}


\begin{proof}[Proof of Lemma \ref{R5}]
Let $\lambda_2=\lambda_2(p)=M_*|p|$ and
\begin{equation*}
\Gamma_q(z, p)=(\mathscr{H}TM_*^{-q}F_q(K^2 z)-K^{2}M_* \langle p, z\rangle)/\lambda_2,
\end{equation*}
then
\begin{equation*}
R_5=K^{4}M_*^2 \sum_{|p|\geqslant A_1K^{1-2q}\mathscr{H}TM_*^{-q-1}} \int \Psi_q(z)e(\lambda_2 \Gamma_q(z, p))\,dz.
\end{equation*}

For $z\in K^{-2}\Omega_q$, we have $D^{\nu}\Psi_q(z)\lesssim K^{-4q|\nu|}$ and $D^{\nu}_z\Gamma_q(z, p)\lesssim 1$. We also have
\begin{equation*}
|\nabla_z \Gamma_q(z, p)+K^2 p/|p||\leqslant K^2/2,
\end{equation*}
which implies $|\nabla_z \Gamma_q(z, p)|\geqslant K^2/2$. By Lemma \ref{app:lemma:2}, we have for any $N\in \mathbb{N}$
\begin{equation*}
\int \Psi_q(z)e(\lambda_2 \Gamma_q(z, p))\,dz\lesssim K^{-(4q+2)N-4}M_*^{-N}|p|^{-N}.
\end{equation*}
The case $N=3$ gives the desired bound for $R_5$.
\end{proof}


\begin{proof}[Proof of Lemma \ref{R6}]
Denote $\lambda_3=K^{-2}\lambda_1$, $g(z, p)=K^2\Phi_q(z, p)$, and
\begin{equation*}
u(z)=[1-\sum_{j=1}^{J(p)}\chi_j(z)] \Psi_q(z).
\end{equation*}

By \eqref{gg}, we have
\begin{equation*}
D^{\nu}_z g(z, p)\lesssim 1.
\end{equation*}
Since $\supp (u)$ is away from critical points, we get $|\nabla_z \Phi_q(z, p)|\gtrsim K^{8q}$ if $z\in \supp (u)$ by \eqref{taylor} (see the proof of Proposition 2.4 in Guo~\cite{guo} for more details), which gives
\begin{equation*}
|\nabla_z g(z, p)|\gtrsim K^{8q+2} \quad \textrm{if $z\in \supp (u)$}.
\end{equation*}
Since $\chi_j$'s have disjoint support and $D^{\nu} \chi_j\lesssim K^{-4q|\nu|}$, we get
\begin{equation*}
D^{\nu} u(z)\lesssim K^{-4q|\nu|}.
\end{equation*}

By Lemma \ref{app:lemma:2} for any $N\in \mathbb{N}$
\begin{equation*}
R_6=\int u(z)e(\lambda_3 g(z, p))\,dz \lesssim K^{-(16q+2)N-4}\lambda_1^{-N}.
\end{equation*}
In particular we get \eqref{boundR6} if we let $N=2$.
\end{proof}

\section{Proof of Theorem \ref{finitetype}}\label{mainterm}

Let $\rho\in C_0^{\infty}(\mathbb{R}^2)$ such that $\int_{\mathbb{R}^2}\rho(y)\,dy=1$. The central question in the lattice point problem is how to estimate the sum
\begin{equation}
\sum_{k\in \mathbb{Z}^2_*} \widehat{\chi}_{\mathcal{B}}(tk)\widehat{\rho}(\varepsilon k). \label{key-sum}
\end{equation}

If $\partial \mathcal{B}$ has positive curvature, by H\"{o}rmander's formula that we mentioned at the beginning of \S \ref{four-tran}, the estimate of \eqref{key-sum} is reduced to an exponential sum. Then one can use the classical van der Corput methods for exponential sums. For such treatment the reader could consult, for example, Kr\"{a}tzel and Nowak~\cite{nowak I, nowak II}, M\"{u}ller~\cite{mullerII}, the author~\cite{guo}, etc.

If the curvature is allowed to vanish, this cannot be done directly. One may then replace $\mathcal{B}$ by $\mathcal{B}_\theta$ and consider the integral of \eqref{key-sum} over all rotations, namely
\begin{equation*}
\int_0^{2\pi} \big|\sum_{k\in \mathbb{Z}^2_*} \widehat{\chi}_{\mathcal{B}_\theta}(t k)\widehat{\rho}(\varepsilon k)\big| \,d\theta.
\end{equation*}
For example Iosevich~\cite{iosevich} estimated such integral by putting absolute value on each term in the sum. Rather than doing that, we properly split the sum into two parts: one with more terms, one with less. We put absolute value on each term in the latter part. To the former part, we apply the asymptotic formula of $\widehat{\chi}_{\mathcal{B}_\theta}(\xi)$ away from those points $\xi$ corresponding to small curvature. This is where we need Corollary \ref{FTofInd} with an error term containing curvature explicitly. Then the estimate is reduced to an exponential sum, to which we can apply similar methods used in \cite{nowak I}, \cite{nowak II}, \cite{mullerII}, and \cite{guo}. The former part is where we gains and the reason why we achieve a sharper bound. We carry out our idea above in the proof of the following lemma.

\begin{lemma}\label{lemma5:0}
Let $\zeta=1/3831$ and $\rho\in C_0^{\infty}(\mathbb{R}^2)$ such that $\int_{\mathbb{R}^2}\rho(y)\,dy=1$. If $\mathcal{B}$ is a compact convex planar domain with smooth boundary of finite type $\omega$ which contains the origin as an interior point, then for $j\in \mathbb{N}$ we have
\begin{equation*}
\int_0^{2\pi} \sup_{2^{j-1}\leqslant t<2^{j+2}} \big|t^{4/3+\zeta+\sigma(\omega)}\sum_{k\in \mathbb{Z}^2_*} \widehat{\chi}_{\mathcal{B}_\theta}(tk)\widehat{\rho}(\varepsilon k)\big| \,d\theta \lesssim 1,
\end{equation*}
where $\varepsilon=\varepsilon(j, \omega)=2^{-j\alpha(\omega)}$,
\begin{equation*}
\alpha(\omega)=\frac{426\omega-832}{1277\omega-2496},\quad \textrm{and} \quad \sigma(\omega)=\frac{832}{1277(1277\omega-2496)}.
\end{equation*}
\end{lemma}

Before we prove this result, we firstly apply it to prove the following lemma, which easily implies Theorem \ref{finitetype} (see Iosevich~\cite[p.~27]{iosevich} for this argument).

\begin{lemma}\label{lemma5:1} Under the same hypothesis as in Lemma \ref{lemma5:0}, for $j\in \mathbb{N}$ we have
\begin{equation*}
\int_0^{2\pi} \sup_{2^j\leqslant t<2^{j+1}} t^{-2/3+\zeta+\sigma(\omega)}|P_{\mathcal{B}_\theta}(t)| \,d\theta \leqslant C,
\end{equation*}
where $C$ is independent of $j$.
\end{lemma}

\begin{proof}[Proof of Lemma \ref{lemma5:1}] Define $\rho_{\varepsilon}(y)=\varepsilon^{-2}\rho(\varepsilon^{-1}y)$ and
\begin{equation*}
N_{\varepsilon, \theta}(t)=\sum_{k\in \mathbb{Z}^2} \chi_{t\mathcal{B}_\theta}*\rho_{\varepsilon}(k).
\end{equation*}
By the Poisson summation formula
\begin{equation*}
N_{\varepsilon, \theta}(t)=t^2\sum_{k\in \mathbb{Z}^2} \widehat{\chi}_{\mathcal{B}_\theta}(tk)\widehat{\rho}(\varepsilon k)=\textrm{area}(\mathcal{B})t^2+R_{\varepsilon, \theta}(t),
\end{equation*}
where
\begin{equation*}
R_{\varepsilon, \theta}(t)=t^2\sum_{k\in \mathbb{Z}^2_*} \widehat{\chi}_{\mathcal{B}_\theta}(tk)\widehat{\rho}(\varepsilon k).
\end{equation*}

M\"uller proved in \cite{mullerI} that if $C_1>0$ satisfies $B(0, 1/C_1)\subset \mathcal{B}$ then
\begin{equation*}
N_{\varepsilon, \theta}(t-C_1\varepsilon)\leqslant \#\{\mathbb{Z}^2\cap t\mathcal{B}_\theta\}=\sum_{k\in \mathbb{Z}^2}\chi_{t\mathcal{B}_\theta}(k)\leqslant N_{\varepsilon, \theta}(t+C_1\varepsilon),
\end{equation*}
which implies
\begin{equation*}
P_{\mathcal{B}_\theta}(t)\lesssim |R_{\varepsilon, \theta}(t+C_1\varepsilon)|+|R_{\varepsilon, \theta}(t-C_1\varepsilon)|+t\varepsilon.
\end{equation*}
Then
\begin{align*}
\sup_{2^j\leqslant t<2^{j+1}} t^{-2/3+\zeta+\sigma(\omega)}|P_{\mathcal{B}_\theta}(t)|  &\lesssim  \sup_{2^j\leqslant t<2^{j+1}} t^{-2/3+\zeta+\sigma(\omega)}t\varepsilon \\
 &\quad +\sup_{2^j\leqslant t<2^{j+1}} t^{-2/3+\zeta+\sigma(\omega)}|R_{\varepsilon, \theta}(t\pm C_1\varepsilon)|.
\end{align*}
The first term on the right side is bounded by a constant, and the second one is in $L^1(S^1)$ due to Lemma \ref{lemma5:0}.
\end{proof}

\begin{proof}[Proof of Lemma \ref{lemma5:0}]
Let $t\in [2^{j-1}, 2^{j+2}]$. For any $\theta \in [0, 2\pi]$ we have the following splitting
\begin{equation*}
\sum_{k\in \mathbb{Z}^2_*} \widehat{\chi}_{\mathcal{B}_\theta}(tk)\widehat{\rho}(\varepsilon k)=\textrm{sum I}+\textrm{sum II},
\end{equation*}
where
\begin{equation*}
\textrm{sum I}=\sum_{k\in D_1(\delta, \theta)} \widehat{\chi}_{\mathcal{B}_\theta}(tk)\widehat{\rho}(\varepsilon k),
\end{equation*}
\begin{equation*}
\textrm{sum II}=\sum_{k\in D_2(\delta, \theta)} \widehat{\chi}_{\mathcal{B}_\theta}(tk)\widehat{\rho}(\varepsilon k),
\end{equation*}
and $D_1(\delta, \theta)$, $D_2(\delta, \theta)$ are two regions defined as follows: if
\begin{equation*}
\delta=\delta(j, \omega)=2^{-j\beta(\omega)}\quad \textrm{with} \quad \beta(\omega)=\frac{\omega-2}{1277\omega-2496},
\end{equation*}
\begin{equation*}
D_2(\delta, 0)=\{\xi\in \mathbb{R}^2_*: K_{\xi}\leqslant \delta \textrm{ or } K_{-\xi}\leqslant \delta \},
\end{equation*}
and $D_1(\delta, 0)=\mathbb{R}^2_{*}\setminus D_2(\delta, 0)$, then
\begin{equation*}
D_2(\delta, \theta)=R_{\theta} D_2(\delta, 0)  \quad \textrm{and} \quad D_1(\delta, \theta)=R_{\theta} D_1(\delta, 0).
\end{equation*}

Note that $D_2(\delta, 0)$ is the union of finitely many planar double cones (symmetric about the origin)\footnote{A planar double cone symmetric about the origin is, for example, the (smaller) region bounded between $y=x$ and $y=1.001x$.} minus the origin. If $t$ is large, these cones intersect only at the origin.

For sum II we have
\begin{claim}\label{claim2}
\begin{equation}
\int_0^{2\pi} \sup_{2^{j-1}\leqslant t<2^{j+2}} t^{4/3+\zeta+\sigma(\omega)}\big| \operatorname{sum\ II}\big| \,d\theta \lesssim 1.\label{claim2-ineq}
\end{equation}
\end{claim}
We defer the proof of this claim until later. Next we will prove
\begin{equation}
\sup_{2^{j-1}\leqslant t<2^{j+2}} t^{4/3+\zeta+\sigma(\omega)}\big| \, \textrm{sum I} \, \big| \lesssim 1 \label{ineq-sumII}
\end{equation}
with an implicit constant depending only on $\mathcal{B}$. The conclusion in Lemma \ref{lemma5:0} follows easily from \eqref{claim2-ineq} and \eqref{ineq-sumII}.

Note if $\xi \in D_1(\delta, \theta)$ then $K_{\pm \xi}^{\theta}\geqslant \delta$. Applying Corollary \ref{FTofInd} to sum I yields
\begin{equation}
\textrm{sum I}=(2\pi)^{-1}e^{3\pi i/4} S_1+(2\pi)^{-1}e^{-3\pi i/4} \widetilde{S}_1+R_1,\label{decomp-sumI}
\end{equation}
where
\begin{equation*}
S_1=S_1(t, \varepsilon, \delta, \theta)=t^{-3/2}\sum_{k\in D_1(\delta, \theta)}|k|^{-3/2}(K_{k}^{\theta})^{-1/2}\widehat{\rho}(\varepsilon k)e(tH_{\theta}(k)),
\end{equation*}
\begin{equation*}
\widetilde{S}_1=\widetilde{S}_1(t, \varepsilon, \delta, \theta)=t^{-3/2}\sum_{k\in D_1(\delta, \theta)}|k|^{-3/2}(K_{-k}^{\theta})^{-1/2}\widehat{\rho}(\varepsilon k)e(-tH_{\theta}(-k)),
\end{equation*}
and
\begin{equation}
R_1\lesssim_{\mathcal{B}} \delta^{-4}t^{-2}\sum_{k\in \mathbb{Z}^2_*}|k|^{-2}|\widehat{\rho}(\varepsilon k)|\lesssim \delta^{-4}t^{-2}\log(\varepsilon^{-1})\lesssim t^{-4/3-\zeta-\sigma(\omega)}.\label{boundR1}
\end{equation}

We will only estimate $S_1$ since $\widetilde{S}_1$ is similar. Denote $\mathscr{C}_1=\{\xi \in \mathbb{R}^2: 1/2 \leqslant |\xi| \leqslant 2 \}$. Let us introduce a dyadic decomposition and a partition of unity.

Assume $\varphi \in C_0^{\infty}(\mathbb{R}^2)$ is a real radial function such
that $\supp(\varphi)\subset \mathscr{C}_1$, $0\leqslant \varphi \leqslant 1$, and
\begin{equation*}
\sum_{l_0=-\infty}^{\infty}\varphi(\frac{\xi}{2^{l_0}})=1  \quad \textrm{for }\xi\in \mathbb{R}^2\setminus \{0\}.
\end{equation*}
Denote
\begin{equation*}
S_{1, M}=\sum_{k\in D_1(\delta, \theta)}\varphi(M^{-1}k)|k|^{-3/2}(K_{k}^{\theta})^{-1/2}\widehat{\rho}(\varepsilon k)e(tH_{\theta}(k)),
\end{equation*}
then $S_1=t^{-3/2}\sum_{l_0=0}^{\infty}S_{1, 2^{l_0}}$. It suffices to estimate $S_{1, M}$ for a fixed $M=2^{l_0}$, $l_0\in \mathbb{N}_0$.

Let $q\in \mathbb{N}$. For each $\xi \in S^1\cap \mathcal{H}$ there exists a cone
\begin{equation*}
\mathfrak{C}(\xi, 2r(\xi)):=\mathop{\cup}\limits_{l>0} l B(\xi, 2r(\xi)),
\end{equation*}
where $r(\xi)=c_2 (K_{\xi})^{4q+2}/2$ and $c_2$ is the constant appearing in the statement of Lemma \ref{non-vanishinglemma}. Note that $K_\eta \asymp K_\xi$ if $\eta\in \mathfrak{C}(\xi, 2r(\xi))$. From the family of cones $\{\mathfrak{C}(\xi, r(\xi)/2): \xi \in S^1\cap \mathcal{H}\}$, we can choose, by a Vitali procedure, a sequence $\{\mathfrak{C}(\xi_i, r(\xi_i)/2)\}_{i=1}^{\infty}$ such that these cones still cover $\mathcal{H}$ and that $\{\mathfrak{C}(\xi_i, r(\xi_i))\}_{i=1}^{\infty}$ satisfy the bounded overlap property. Denote
\begin{equation*}
\mathfrak{C}_i^{\theta}=R_{\theta}\mathfrak{C}(\xi_i, r(\xi_i)).
\end{equation*}
Then $\{\mathfrak{C}_i^{\theta}\}_{i=1}^{\infty}$ forms an open cover of $\mathcal{H}_{\theta}$. We can construct a partition of unity $\{\psi_i\}_{i=1}^{\infty}$ such that
\begin{enumerate}[\upshape (i)]
\item  $\sum_{i}\psi_i \equiv 1$ on $\mathcal{H}_{\theta}$, and $\psi_i\in C_0^{\infty}(\mathfrak{C}_i^{\theta})$;
\item  each $\psi_i$ is homogeneous of degree zero;
\item  $|D^{\nu}\psi_i|\lesssim (K_{\xi_i})^{-(4q+2)|\nu|}$ on $\mathscr{C}_1$.
\end{enumerate}

From the family $\{\mathfrak{C}_i^{\theta}\}_{i=1}^{\infty}$ we can find a subfamily $\{\mathfrak{C}_i^{\theta}\}_{i\in \mathscr{A}}$ covering $D_1(\delta, \theta)$, where $\mathscr{A}=\mathscr{A}(\delta)$ is an index set such that $i\in \mathscr{A}$ if and only if $\mathfrak{C}_i^{\theta}$ intersects $D_1(\delta, \theta)$. Since $r(\xi_i)\gtrsim \delta^{4q+2}$ for any $i\in \mathscr{A}$, a size estimate gives that $\#\mathscr{A}\lesssim \delta^{-4q-2}$. Define
\begin{equation}
S_{1, M}^*=\sum_{i\in \mathscr{A}}S_{2, i},\label{formula1}
\end{equation}
where
\begin{equation*}
S_{2, i}=\sum_{k\in \mathbb{Z}^2} U_i^{\theta}(k) e(tH_{\theta}(k))
\end{equation*}
and
\begin{equation*}
U_i^{\theta}(k)=\psi_i(M^{-1}k)\varphi(M^{-1}k)|k|^{-3/2}(K_{k}^{\theta})^{-1/2}\widehat{\rho}(\varepsilon k).
\end{equation*}

Instead of $S_{1, M}$ we will estimate $S_{1, M}^*$. It turns out that the error
\begin{equation}
R_{2, M}=S_{1, M}^*-S_{1, M} \label{R2M}
\end{equation}
is relatively small and this will be clear later (see Claim \ref{claim3} below). To estimate $S_{1, M}^*$, it suffices to estimate $S_{2, i}$ for any fixed $i\in \mathscr{A}$.

By Lemma \ref{non-vanishinglemma} and the homogeneity of $H_{\theta}$, there exist two orthogonal vectors $v_1^*=v_1^*(R_{\theta}\xi_i)$, $v_2^*=v_2^*(R_{\theta}\xi_i)$$\in \mathbb{Z}^2$ such that $|v_1^*|=|v_2^*|\asymp (K_{\xi_i})^{-4q}$, $\|(v_1^*, v_2^*)^{-1}\|\lesssim (K_{\xi_i})^{4q}$, and
\begin{equation}
|h_q^{\theta}(\eta, v_1^*, v_2^*)|\gtrsim (K_{\xi_i})^{-8q^2-16q-2} \quad \textrm{if} \ \eta \in \mathop{\cup}\limits_{1/4\leqslant l\leqslant 4} l B(R_{\theta}\xi_i, 2r(\xi_i)).\label{determinant}
\end{equation}
Let $L=[\mathbb{Z}^2:\mathbb{Z}v_1^*\oplus\mathbb{Z}v_2^*]$ be the index of the lattice spanned by $v_1^*$, $v_2^*$ in the lattice $\mathbb{Z}^2$. Then
\begin{equation*}
L=|\det(v_1^*, v_2^*)|\lesssim (K_{\xi_i})^{-8q},
\end{equation*}
and there exist vectors $b_l\in \mathbb{Z}^2$ ($l=1, \ldots, L$) such that $|b_l|\lesssim (K_{\xi_i})^{-4q}$ and
\begin{equation*}
\mathbb{Z}^2=\uplus_{l=1}^{L}(\mathbb{Z}v_1^*+\mathbb{Z}v_2^*+b_l).
\end{equation*}
Let $N\in \mathbb{N}$ be arbitrarily fixed. Applying this decomposition, for any $k\in \mathbb{Z}^2$ we can write $k=m_1v_1^*+m_2v_2^*+b_l$ where $m_s\in \mathbb{Z}$ ($s=1, 2$). Hence
\begin{align}
S_{2, i}&=\sum_{l=1}^{L}\sum_{m\in \mathbb{Z}^2}U_i^{\theta}(m_1v_1^*+m_2v_2^*+b_l) e(tH_{\theta}(m_1v_1^*+m_2v_2^*+b_l))\nonumber\\
        &=(K_{\xi_i})^{-1/2}M^{-3/2}(1+M\varepsilon)^{-N}\sum_{l=1}^{L}S(T, M_*; G_l, F_l),\label{formula2}
\end{align}
where $T=tM$, $M_*=K^{4q}M$, $K=K_{\xi_i}$,
\begin{equation*}
G_l(y)=K^{1/2}M^{3/2}(1+M\varepsilon)^{N} U_i^{\theta}(M_*y_1v_1^*+M_*y_2v_2^*+b_l),
\end{equation*}
and
\begin{equation*}
F_l(y)=H_{\theta}(y_1K^{4q}v_1^*+y_2K^{4q}v_2^*+M^{-1}b_l).
\end{equation*}

We consider the function $F_l$ restricted to the convex domain
\begin{equation*}
\Omega_l=\{(y_1, y_2)^{t}\in \mathbb{R}^2 : y_1K^{4q}v_1^*+y_2K^{4q}v_2^*+M^{-1}b_l\in \mathop{\cup}\limits_{1/4\leqslant l\leqslant 4} l B(R_{\theta}\xi_i, 2r(\xi_i)) \}.
\end{equation*}
The support of $G_l$ satisfies
\begin{equation*}
\supp(G_l)\subset \{(y_1, y_2)^{t}\in \mathbb{R}^2 : y_1K^{4q}v_1^*+y_2K^{4q}v_2^*+M^{-1}b_l\in \overline{\mathscr{C}_1\cap \mathfrak{C}_i^{\theta}} \}\subset \Omega_l.
\end{equation*}

We want to apply to $S(T, M_*; G_l, F_l)$ Proposition \ref{claim1} with $G=G_l$, $F=F_l$, $\Omega=\Omega_l$, and $q=3$.

Since $1 \gtrsim K_{\xi_i}\gtrsim \delta$ for $i\in \mathscr{A}$, there exist positive constants $C_2$ and $C_3$ such that the assumptions of Proposition \ref{claim1} are satisfied if $C_3\delta^{-26}\leqslant M\leqslant C_2 t^{4/5}$. This follows from Lemma \ref{non-vanishinglemma}, \eqref{determinant} and the following facts: if $(K_{\xi_i})^{-4q}\lesssim M$ then $\Omega_l\subset c_0 B(0, 1)$ for a constant $c_0$ (depending only on $q$, $\mathcal{B}$);
\begin{equation*}
\dist \left(\big(\mathop{\cup}\limits_{1/4\leqslant l\leqslant 4} l B(R_{\theta}\xi_i, 2r(\xi_i)) \big)^c \, , \, \overline{\mathscr{C}_1\cap \mathfrak{C}_i^{\theta}}\right)\geqslant c_2 (K_{\xi_i})^{4q+2}/8;
\end{equation*}
and
\begin{equation*}
D^{\nu}U_i^{\theta} \lesssim (K_{\xi_i})^{-(4q+2)|\nu|-1/2}M^{-|\nu|-3/2}(1+M\varepsilon)^{-N}.
\end{equation*}

Thus by Proposition \ref{claim1} (with $q=3$) we get
\begin{equation}
S(T, M_*; G_l, F_l)\lesssim_{\mathcal{B}} (K_{\xi_i})^{24-97/22}t^{1/22}M^{20/11}+R, \label{kk}
\end{equation}
where
\begin{align*}
R\lesssim_{\mathcal{B}} K_{\xi_i}^{24}&\big[K_{\xi_i}^{-16}M^{7/4}+K_{\xi_i}^{-921/88}t^{-15/88}M^{24/11}\\
                  &\quad+K_{\xi_i}^{-145/22}t^{-1/22}M^{85/44}(\log M)^{1/8}\big].
\end{align*}
Using \eqref{formula1}, \eqref{formula2}, \eqref{kk}, $K_{\xi_i}\gtrsim \delta$, and bounds for $\#\mathscr{A}$ and $L$, we get
\begin{equation}
\begin{split}
S_{1, M}^* &\lesssim \delta^{-29/2}M^{-3/2}(1+M\varepsilon)^{-N}\big[\delta^{-97/22}t^{1/22}M^{20/11}+\delta^{-16}M^{7/4}\\
&\quad+\delta^{-921/88}t^{-15/88}M^{24/11}+\delta^{-145/22}t^{-1/22}M^{85/44}(\log M)^{1/8}\big].
\end{split}\label{for-S1-1}
\end{equation}

Now we can estimate $S_1$. By \eqref{R2M},
\begin{equation}
S_1=t^{-3/2}\left( \sum_{C_3\delta^{-26}\leqslant 2^{l_0}\leqslant C_2t^{4/5}}S_{1, 2^{l_0}}^*+R_2+R_3+R_4\right),\label{boundS1}
\end{equation}
where
\begin{equation*}
R_2=-\sum_{C_3\delta^{-26}\leqslant 2^{l_0}\leqslant C_2 t^{4/5}}R_{2, 2^{l_0}},
\end{equation*}
\begin{equation*}
R_3=\sum_{2^{l_0}<C_3\delta^{-26}}S_{1, 2^{l_0}}, \quad \textrm{and} \quad R_4=\sum_{2^{l_0}> C_2 t^{4/5}}S_{1, 2^{l_0}}. \end{equation*}
Taking \eqref{for-S1-1} and sizes of $\delta$ and $\varepsilon$ into account, we get
\begin{align}
\sum_{C_3 \delta^{-26}\leqslant 2^{l_0}\leqslant C_2t^{4/5}}S_{1, 2^{l_0}}^*
     &\lesssim_{\mathcal{B}} \delta^{-208/11}t^{1/22}\varepsilon^{-7/22}+\delta^{-2197/88}t^{-15/88}\varepsilon^{-15/22}\nonumber\\
     &\quad\quad +\delta^{-61/2}\varepsilon^{-1/4}+\delta^{-232/11}
t^{-1/22}\varepsilon^{-19/44}\log t \nonumber\\
     &\lesssim_{\mathcal{B}} \delta^{-208/11}t^{1/22}\varepsilon^{-7/22}.\label{qq}
\end{align}

\begin{claim}\label{claim3}
\begin{equation*}
\max(|R_2|, |R_3|, |R_4|)\lesssim_{\mathcal{B}} \delta^{-208/11}t^{1/22}\varepsilon^{-7/22}.
\end{equation*}
\end{claim}

Hence \eqref{ineq-sumII} follows from \eqref{decomp-sumI}, \eqref{boundR1}, \eqref{boundS1}, \eqref{qq}, Claim \ref{claim3}, and sizes of $\delta$ and $\varepsilon$.
\end{proof}


\begin{remarks}
(1) Our proof works for convex planar domains of finite type. If $\omega=2$, the curvature does not vanish and $D_2(\delta, \theta)$ is empty. The method we used is essentially the same as those used in \cite{mullerII}, \cite{guo}, and will produce the same bound $O(t^{2/3-1/87})$.

(2) With essentially the same proof, we can actually prove that if $1\leqslant p< 2+2/(\omega-2)$ then
\begin{equation*}
\sup\limits_{t\geqslant 2} \log^{-a}(t)t^{-2/3+\Upsilon(\omega, p)}|P_{\mathcal{B}_\theta}(t)| \in L^{p}(S^1)
\end{equation*}
where $a>1/p$ and
\begin{equation*}
\Upsilon(\omega, p)=\frac{(2-p)\omega+2p-2}{3(1219p+58)\omega-3(2438p+58)}.
\end{equation*}

(3) We can possibly improve the exponent in Theorem \ref{ae-theorem} by iterating the van der Corput method.
\end{remarks}


\begin{proof}[Proof of Claim \ref{claim2}]
To begin with, we estimate
\begin{align*}
(\ast)&:=\int_{0}^{2\pi} 1_{D_2(\delta, \theta)}(k) \big(\sup_{2^{j-1}\leqslant t<2^{j+2}} |t k|^{3/2}|\widehat{\chi}_{\mathcal{B}_\theta}(t k)|\big) \,d\theta\\
      &=\int_{0}^{2\pi} 1_{D_2(\delta, 0)}(|k|(\cos \theta, \sin \theta)) \big(\sup_{t \sim 2^{j}} |t k|^{3/2}|\widehat{\chi}_{\mathcal{B}}(t|k|(\cos \theta, \sin \theta))|\big) \,d\theta.
\end{align*}
Recall the definition of $D_2(\delta, 0)$, we are only integrating over finitely many (no more than $2\Xi$) small arcs on $S^1$. By Lemma \ref{lemma2:1} the length of each arc is $\lesssim \delta^{(\omega_i-1)/(\omega_i-2)}$ ($i=1,\ldots, \Xi$). The inequality \eqref{randolbound} gives explicit upper bounds for
\begin{equation*}
\sup |tk|^{3/2}|\widehat{\chi}_{\mathcal{B}}(t|k|(\cos \theta, \sin \theta))|
\end{equation*}
over these small arcs. Hence the estimate is reduced to the following integral
\begin{equation*}
\int_0^{c\delta^{(\omega_i-1)/(\omega_i-2)}} \theta^{-\frac{\omega_i-2}{2(\omega_i-1)}} \, d\theta \lesssim \delta^{1/2+1/(\omega_i-2)}.
\end{equation*}
It follows that
\begin{equation*}
(\ast)\lesssim \sum_{i=1}^{\Xi} \delta^{1/2+1/(\omega_i-2)} \lesssim \delta^{1/2+1/(\omega-2)}.
\end{equation*}

Hence the left side of \eqref{claim2-ineq} is bounded by
\begin{align*}
&\lesssim (2^j)^{-1/6+\zeta+\sigma(\omega)} \sum_{k\in \mathbb{Z}^2_*} |\widehat{\rho}(\varepsilon k)||k|^{-3/2} (\ast)\\
          &\lesssim (2^j)^{-1/6+\zeta+\sigma(\omega)} \varepsilon^{-1/2}\delta^{1/2+1/(\omega-2)} \lesssim 1.
\end{align*}
In the last step, we use the definition of $\varepsilon$ and $\delta$.
\end{proof}


\begin{proof}[Proof of Claim \ref{claim3}]
We firstly estimate $R_2$. The trivial estimate gives
\begin{equation*}
|R_{2, M}|\leqslant \sum_{k\in \mathbb{Z}^2\setminus D_1(\delta, \theta)}\sum_{i\in \mathscr{A}} |\psi_i(k/M)||\varphi(k/M)||k|^{-3/2}(K_{k}^{\theta})^{-1/2}|\widehat{\rho}(\varepsilon k)|.
\end{equation*}
If $k\in (\mathbb{Z}^2\setminus D_1(\delta, \theta))\cap \mathfrak{C}_i^{\theta}$ for any $i\in \mathscr{A}$, there are only two possibilities as to the size of $K_k^{\theta}$: (i) If $K_k^{\theta}\gtrsim 1$ then $k$ is contained in finitely many (no more than $\Xi$) cones with angles $\lesssim \delta^{(\omega-1)/(\omega-2)}$; (ii) If $\delta \lesssim K_{k}^{\theta}\lesssim  \delta^{1/2}$ then $k$ is contained in $2\Xi$ cones with angles $\lesssim \delta^{2q+1}$. Based on these two cases, we split the sum above into two parts as follows:
\begin{equation*}
|R_{2, M}|\leqslant \sum_{\substack{k\in \mathbb{Z}^2\setminus D_1(\delta, \theta)\\K_k^{\theta}\gtrsim 1}}\sum_{i\in \mathscr{A}}+\sum_{\substack{k\in \mathbb{Z}^2\setminus D_1(\delta, \theta)\\ \delta \lesssim K_k^{\theta}\lesssim  \delta^{1/2}}}\sum_{i\in \mathscr{A}}.
\end{equation*}
Using this splitting, we get
\begin{align*}
|R_2|&\leqslant \sum_{C_3 \delta^{-26}\leqslant 2^{l_0}\leqslant C_2 t^{4/5}}|R_{2, 2^{l_0}}|\lesssim \delta^{(\omega-1)/(\omega-2)}\varepsilon^{-1/2}
+\delta^{2q+1/2}\varepsilon^{-1/2}\\
     &\lesssim \delta^{-208/11}t^{1/22}\varepsilon^{-7/22}.
\end{align*}

As to $R_3$, by a trivial estimate of $S_{1, 2^{l_0}}$, we get
\begin{equation*}
|R_3|\leqslant \delta^{-1/2}\sum_{2^{l_0}<C_3 \delta^{-26}}\sum_{k\in \mathbb{Z}^2_*}\varphi(2^{-l_0}k)|k|^{-3/2}|\widehat{\rho}(\varepsilon k)|  \lesssim \delta^{-27/2}.
\end{equation*}
Similarly, $R_4=O(\delta^{-1/2}t^{2/5-4N/5}\varepsilon^{-N})$ for any $N\in \mathbb{N}$. Both are smaller than $\delta^{-208/11}t^{1/22}\varepsilon^{-7/22}$.
\end{proof}


\appendix
\section{Several Lemmas}\label{app1}

Here is a quantitative version of the inverse function theorem.

\begin{lemma}\label{app:lemma:1}
Suppose $f$ is a $C^{(k)}$ ($k\geqslant 2$) mapping from an open set $\Omega\subset\mathbb{R}^d$
into $\mathbb{R}^d$ and $b=f(a)$ for some $a\in \Omega$. Assume $|\det (\nabla f(a))|$ $\geqslant$ $c$ and for any $x\in\Omega$,
\begin{equation*}
|D^{\nu} f_i(x)|\leqslant C
\quad \quad \textrm{for $|\nu|\leqslant 2$, $1\leqslant i\leqslant d$}.
\end{equation*}
If $r_0\leqslant \sup\{r>0: B(a, r)\subset \Omega\}$, then $f$ is bijective from $B(a, r_1)$ to an open set containing $B(b, r_2)$ where
\begin{equation*}
r_1=\min\{\frac{c}{2d^{7/2} (d-1)! C^d}, r_0\},
\end{equation*}
\begin{equation*}
r_2=\frac{c}{4d^{3/2}(d-1)!C^{d-1}}r_1.
\end{equation*}
The inverse mapping $f^{-1}$ is also in $C^{(k)}$.
\end{lemma}

H\"ormander~\cite{hormander} Theorem 7.7.1 gives the following estimate obtained by integration by parts.

\begin{lemma}\label{app:lemma:2}
Let $K\subset \mathbb{R}^d$ be a compact set, $X$ an open neighborhood of $K$ and $k$ a nonnegative integer. If $u\in C_0^k(K)$, real $f\in C^{k+1}(X)$, then
\begin{equation*}
|\int u(x) e^{i \lambda f(x)}\,dx|\leqslant C |K|\lambda^{-k} \sum_{|\nu|\leqslant k} \sup |D^{\nu}u||\nabla f|^{|\nu|-2k}, \quad \lambda>0.
\end{equation*}
Here $C$ is bounded when $f$ stays in a bounded set in $C^{k+1}(X)$.
\end{lemma}

The following two lemmas are two forms of the method of stationary phase. The first one is the one dimensional version of H\"ormander~\cite{hormander} Lemma 7.7.3. The second is Sogge and Stein~\cite{soggestein} Lemma 2.

\begin{lemma}\label{app:lemma:4} If $u\in \mathscr{S}(\mathbb{R})$ then for every $k\in \mathbb{N}$
\begin{align*}
\big|\int u(x) e^{-i \lambda x^2/2} \,dx &-(2\pi)^{1/2}e^{-\pi i/4}\lambda^{-1/2}\sum_{j=0}^{k-1}(2i\lambda)^{-j} u^{(2j)}(0)/j!\big|\\
  &\leqslant (2^{1-k}\sqrt{\pi}/k!) \lambda^{-k-1/2} (\|u^{(2k)} \|_{L^2}+\|u^{(2k+1)} \|_{L^2}).
\end{align*}
\end{lemma}

\begin{lemma}\label{app:lemma:3}
Suppose $\phi$ and $\psi$ are smooth functions in $B(0, \delta)$ with $\phi$ real-valued. Assume that
\begin{equation*}
|(\frac{\partial}{\partial x})^{\nu}\phi|\leqslant C \quad \textrm{for $|\nu|\leqslant d+2$}
\end{equation*}
and
\begin{equation*}
|(\frac{\partial}{\partial x})^{\nu}\psi|\leqslant C\delta^{-|\nu|} \quad \textrm{for $|\nu|\leqslant d$}.
\end{equation*}
We also suppose that $(\nabla \phi)(0)=0$, but $|\det \nabla^2\phi(0)|\geqslant \delta$. Then there exists a positive constant $c_1$ (depending only on $\phi$), which is sufficiently small, so that if $\psi$ is supported in $B(0, c_1\delta)$ we can assert that
\begin{equation*}
|\int \psi e^{i\lambda \phi}\,dx|\leqslant C\lambda^{-d/2}\delta^{-1/2}.
\end{equation*}
\end{lemma}


\subsection*{Acknowledgments}
The subject of this paper was suggested by Professor Andreas Seeger. I would like to express my gratitude to him for his valuable advice and great help during the work.


\end{document}